\documentclass[preprint,12pt]{elsarticle}

\usepackage{amssymb}
\usepackage{amsmath}
\usepackage{amsthm}
\usepackage{amsfonts}
\usepackage[utf8]{inputenc}
\usepackage[T1]{fontenc}
\usepackage[dvips,ps2pdf]{hyperref}
\hypersetup{
    colorlinks = true,
    citecolor = blue,
    linkcolor = red,
    backref = true,
    breaklinks = true,
}

\newtheorem{remark}{Remark}[section]

\newtheorem{theorem}{Theorem}[section]

\newtheorem{lemma}{Lemma}[section]
\newtheorem{props}{Proposition}[section]
\newtheorem{corollary}{Corollary}[section]

\makeatletter
\def\ps@pprintTitle{%
  \let\@oddhead\@empty
  \let\@evenhead\@empty
  \let\@oddfoot\@empty
  \let\@evenfoot\@oddfoot
}
\makeatother

\begin{document}

\begin{frontmatter}


\title{Dynamic Programming Principle for Stochastic Control Problems driven by General L\'{e}vy Noise}
\author[label1]{Ben Goldys}
\author[label2]{Wei Wu\corref{cor1}}
\ead{wei.wu.0519@gmail.com}
\cortext[cor1]{Corresponding author}
\address[label1]{School of Mathematics and Statistics, The University of Sydney, NSW 2006 Australia}
\address[label2]{School of Mathematics and Statistics, UNSW Australia UNSW Sydney, NSW 2052 Australia}


\begin{abstract}
   We extend the proof of the dynamic programming principle (DPP) for standard stochastic optimal control problems driven by general L\'{e}vy noises. 
   Under appropriate assumptions, it is shown that the DPP still holds when the state process fails to have any moments at all.
\end{abstract}


\end{frontmatter}


\section{Introduction}\label{intro}
The dynamic programming principle (DPP) is a well-known device in studying stochastic optimal control problems. For a standard control problem with finite 
horizon, it states that the value function for the control problem starting at time $s \in [0,T]$ from a position $X_{s} = x$ is given by the formula
\begin{equation}\label{eqn:dppjump1}
   V(s,x) = \sup_{u \in \mathcal{A}_{s}}\mathbb{E}\Big(\int_{s}^{\tau}f(r,X_{r},u_{r})dr + V(\tau,X_{\tau})\Big),
\end{equation}
where $\tau$ is some stopping time, $u$ is an admissible control process, $\mathcal{A}_{s}$ is a given admissible control set at time $s$, and $X$ is a 
controlled state process.  All terms will be defined in a more precise way later. \\

There are many ways to prove the DPP. When the underlying probability space is fixed in advance, we say a stochastic control problem is under a strong 
formulation. In this case, one may use the theory of piecewise constant controls to construct appropriate supermartingales and show that the DPP holds 
through properties of supermartingales (see \cite{Kry80} for the diffusion case, and \cite{Ish04} for the jump case). Alternatively, we can prove the DPP 
by partitioning the state space, provided the value function satisfies certain regularity conditions or using its semicontinuous envelope (see for example 
\cite{BouT11}). When a control problem is defined in the weak sense, that is the underlying probability space is taken to be part of the control, we can 
also apply this approach (see for example \cite{AzePW14,YonZ99}). Moreover, recently by interpreting controls in the weak sense, \cite{ElkT13} proved the 
DPP by using a probabilistic approach.  \\

To prove the DPP, in most cases, the state process is required to have finite second moments (see for example \cite{AzePW14,BouT11,Ish04,Kry80,YonZ99}). A 
stochastic control problem is often formulated with the state process assumed to follow a certain stochastic differential equation (SDE). For SDE driven by 
Brownian noise, with appropriate assumptions on the coefficients of SDE, it is well known that the existence of finite second moments is assured. However, 
this does not hold in general case when the SDE is driven by a more general L\'{e}vy type noise. For example, let us consider the following (controlled)
state process: 
\begin{eqnarray*}
  \left\{\begin{array}{ll} dX_{t} = b(t,X_{t-},u_{t})dt + \sigma(t,X_{t-},u_{t})dW_{t} 
                                    + \displaystyle\int_{0 < |\eta| < 1}\gamma(t,X_{t-},u_{t},\eta)\tilde{N}(dt,d\eta) \\
                           \ \ \ \ \ \ \ \ \ + \displaystyle\int_{|\eta| \geq 1}\gamma(t,X_{t-},u_{t},\eta)N(dt,d\eta)                \\
                            \ \ \ \ \ \ \ \ \ \ \ \ \ \ \ \ \ \ \ \ \ \ \  \ \ \ \ \ \ \ \ \ \ \ \ \ \ \ \ \ \ \ \ \  \ \ \ \ \ \ \ \ \ \ \        \\
                           X_{s} = x \in \mathbb{R}^{d},  \ \ \ \ \ \ \ \ \ \ \ \ \ \ \ \ \ \ \ \ \ \ \  \ \ \ \ \ \ \ \ \ \ \ \ \ \ \ \ \ \ \ \ \
                                                             \ \ \ \ \ \ \ \ \ \ \ \ 0 \leq s \leq t \leq T, \end{array}\right.
\end{eqnarray*}
where $W$ is a Brownian motion, $N$ is a Poisson random measure, and $\tilde{N}$ is the associated compensated Poisson random measure. In this case, we need 
further assumption on the measure $\nu$, for example 
\begin{eqnarray}\label{eqn:aspm}
  \displaystyle\int_{|\eta| \geq 1}|\eta|^{p}\nu(d\eta)dt  &  <  &  \infty, \ \textnormal{for some} \ p \geq 2,
\end{eqnarray}
to assure that there exists a finite second moment for the state process. This would restrict us to only a subclass of L\'{e}vy type noises. However, in order 
to study state processes with heavy tailed distributions, one needs to relax the moments assumption. \\

Z\u{a}linescu extended the proof of DPP to stable processes, which requires (\ref{eqn:aspm}) to hold for a certain $p > 0$. He proved the DPP in the context of 
a combined control and optimal stopping problem in which a $C^{2}$-approximation of the state process is introduced. In contrast, the recent work of \cite{ElkT13} 
formulates the stochastic control problems in terms of controlled martingale problems. Their proof assumes that (\ref{eqn:aspm}) holds for $p = 1$. In this work,
our main contribution is to extend the proof of DPP (under the strong formulation) by relaxing (\ref{eqn:aspm}) in which no finite moments assumption are imposed. 
To this end, we use an approximation of state process which is commonly used in construction of solution of SDEs (see for example Theorem 6.2.9 on p374 in \cite{App09} 
or pp354-355 in \cite{Kun04}). The idea behind this is to define a new state process by cutting off the jumps if they are too 'large'. Since 'large' jumps cause the 
failure of the existence of moments, by cutting of the 'large' jumps we retain the nice property of existence of moments. In contrast to the $C^{2}$-approximation
which considered by Z\u{a}linescu, the approximation which we consider simplifies the proof of the DPP. \\

The paper is organized in the following way. In section 2, we formulate our control problem and state the DPP. In section 3, we present an approximation of the 
state process as well as some auxiliary results. Finally, we prove the DPP in section 4.

\section{Problem Formulation}
We will work on the Wiener-Poisson space. Let us recall the construction of such a space given in \cite{BouT11,IshK06}. To this end, 
we first recall the definitions of Wiener and Poisson spaces. Fix a $T > 0$. Let $\Omega_{W} = C([0,T];\mathbb{R}^{d})$, and for 
$\omega_{1} \in \Omega_{W}$, set $W_{t}(\omega_{1}) := \omega_{1}(t)$. Define $\mathbb{F}^{W} := (\mathcal{F}^{W}_{t})_{t \geq 0}$ 
as the smallest filtration such that $W_{s}$ is measurable with respect to $\mathcal{F}^{W}_{t}$ for all $s \in [0,t]$. On 
$(\Omega_{W},\mathcal{F}^{W})$, let $\mathbb{P}_{W}$ be the probability measure such that $W$ is the $m$-dimensional standard Brownian 
motion, where $\mathcal{F}^{W} = \mathcal{F}^{W}_{T}$. Then, we obtain the Wiener space $(\Omega_{W},\mathcal{F}^{W},\mathbb{P}_{W})$.
Let $\mathbb{R}^{q}_{0} =  \mathbb{R}^{q}\backslash\{0\}$, and $\Omega_{N}$ be the set of integer-valued measures on $[0,T] \times 
\mathbb{R}^{q}_{0}$. For $\omega_{2} \in \Omega_{N}$, set $N(\omega_{2},I \times A) := \omega_{2}(I \times A)$, where 
$I \in \mathcal{B}([0,t])$, and $A \in \mathcal{B}(\mathbb{R}^{q}_{0})$. Define $\mathbb{F}^{N} := (\mathcal{F}^{N}_{t})_{t \geq 0}$ 
as the smallest filtration such that $N(\cdot, I \times A)$ is measurable with respect to $\mathcal{F}^{N}_{t}$ for all $I \in 
\mathcal{B}([0,t])$ and $A \in \mathcal{B}(\mathbb{R}^{q}_{0})$. On $(\Omega_{N}, \mathcal{F}^{N})$, let $\mathbb{P}_{N}$ be the 
probability measure such that $N$ is the Poisson random measure with intensity $\nu$, where $\mathcal{F}^{N} = \mathcal{F}^{N}_{T}$, 
and $\nu$ is the L\'{e}vy measure, i.e., it satisfies
\begin{eqnarray}
  \displaystyle\int_{\mathbb{R}^{q}_{0}} (|\eta|^{2} \wedge 1) \nu(d\eta) < \infty. \nonumber
\end{eqnarray}
Then, we obtain the Poisson space $(\Omega_{N},\mathcal{F}^{N},\mathbb{P}_{N})$. Now, consider the product space $\Omega = \Omega_{W} \times 
\Omega_{N}$. For $\omega = (\omega_{1},\omega_{2}) \in \Omega$, set $W_{t}(\omega) := W_{t}(\omega_{1})$, and $N(\omega,I \times A) := 
N(\omega_{2},I \times A)$.  Let $\mathbb{P} := \mathbb{P}_{W} \otimes \mathbb{P}_{N}$ be the probability measure on $(\Omega,\mathcal{F})$,
where $\mathcal{F}$ is the completion of $\mathcal{F}^{W} \otimes \mathcal{F}^{N}$. This then yields the Wiener-Poisson space 
$(\Omega,\mathcal{F},\mathbb{P})$. Without of loss generality, we may assume that this space is complete. On this space, we may associate
a filtration $(\mathcal{F}_{t})_{t \geq 0}$ which is the right-continuous completed revision of the filtration 
$(\mathcal{F}^{W}_{t} \otimes \mathcal{F}^{N}_{t})_{t \geq 0}$.  \\
 
Let $\mathcal{F}^{W,s}_{t}$ be the smallest $\sigma$-algebra such that $W_{r} - W_{s}$ is measurable with respect to $\mathcal{F}^{W,s}_{t}$ 
for all $r \in [s, t \vee s]$, and $\mathcal{F}^{N,s}_{t}$ be the smallest $\sigma$-algebra such that $N(\cdot,I_{2} \times A) - 
N(\cdot, I_{1} \times A)$ is measurable with respect to $\mathcal{F}^{N,s}_{t}$ for all $I_{1}, I_{2} \in \mathcal{B}([s,t \vee s])$, $A \in 
\mathcal{B}(\mathbb{R}^{q}_{0})$, where $I_{1} \subset I_{2}$. We define a commonly used filtration $(\mathcal{F}^{s}_{t})_{t \geq s}$ which
is the right-continuous completed revision of $(\mathcal{F}^{W,s}_{t} \otimes \mathcal{F}^{N,s}_{t})_{t \geq s}$ (see for example \cite{BouT11} 
for this filtration). For the sake of notations, from now on, we write $N(\omega, (0,t] \times A)$ as $N(t,A)$.  \\

Next, we consider the following control problem. Fix $s \in [0,T)$, the controlled state process $(X_{t})_{t \geq s}$ is assumed to follow 
the SDE:
\begin{eqnarray}\label{eqn:dynsdelarge}
    \left\{\begin{array}{ll} dX_{t} = b(t,X_{t-},u_{t})dt + \sigma(t,X_{t-},u_{t})dW_{t} 
                                                      + \displaystyle\int_{0 < |\eta| < 1}\gamma(t,X_{t-},u_{t},\eta)\tilde{N}(dt,d\eta) \\
                                    \ \ \ \ \ \ \ \ \ + \displaystyle\int_{|\eta| \geq 1}\gamma(t,X_{t-},u_{t},\eta)N(dt,d\eta)                                  \\
                            \ \ \ \ \ \ \ \ \ \ \ \ \ \ \ \ \ \ \ \ \ \ \  \ \ \ \ \ \ \ \ \ \ \ \ \ \ \ \ \ \ \ \ \  \ \ \ \ \ \ \ \ \ \ \                      \\
                           X_{s} = x \in \mathbb{R}^{d},  \ \ \ \ \ \ \ \ \ \ \ \ \ \ \ \ \ \ \ \ \ \ \  \ \ \ \ \ \ \ \ \ \ \ \ \ \ \ \ \ \ \ \ \
                                                             \ \ \ \ \ \ \ \ \  0 \leq s \leq t \leq T, \end{array}\right.
\end{eqnarray}
where $X_{t-}$ is the left limit of $X_{t}$, and $u: [0,T] \times\Omega \rightarrow \mathbb{R}^{\ell}$ is a predictable process which acts as a control. 
Moreover, $b: [0,T] \times \mathbb{R}^{d} \times \mathbb{R}^{\ell} \rightarrow  \mathbb{R}^{d}$ is a continuous function, $\sigma: [0,T] \times\mathbb{R}^{d}
\times\mathbb{R}^{\ell} \rightarrow \mathbb{R}^{d \times m}$ is a continuous function, $\gamma: [0,T]\times\mathbb{R}^{d}\times\mathbb{R}^{\ell}\times
\mathbb{R}^{q}_{0} \rightarrow \mathbb{R}^{d}$ is a Borel measurable function, and $\gamma$ is continuous in $(t,x,u)$ for every $\eta \in \mathbb{R}^{q}_{0}$. 
In addition, $\tilde{N}$ is the compensated Poisson random measure associate to $N$, i.e. $\tilde{N}(dt,d\eta) = N(dt,d\eta) - \nu(d\eta)dt$. \\
   
Fix a compact set $A \subset \mathbb{R}^{\ell}$. The set of admissible controls $(u_{t})_{t \in [0,T]}$ is denoted by $\mathcal{A}_{s}$, where
\begin{eqnarray*}
   \mathcal{A}_{s}   & = & \Big\{u: [0,T] \times \Omega \rightarrow A \ \textnormal{is predictable with respect to}  \ (\mathcal{F}^{s}_{t})_{t \geq 0}\Big\}. 
\end{eqnarray*}

In the rest of the paper, we shall make the following assumption. \\

\hypertarget{assump1}{\textbf{Assumption 1}.}  There exist constants $C>0$ and $C_{M} > 0$ such that for all $t \in [0,T]$, $u \in A$, $x_{1}, x_{2} \in 
                                               \mathbb{R}^{d}$, and $0 < |\eta| < M$, we have                                               
                                               \begin{eqnarray}
                                                  \big|\sigma(t,x_{1},u) - \sigma(t,x_{2},u)\big| + \big|b(t,x_{1},u) - b(t,x_{2},u)\big|  
                                                                                                           &  \leq  &  C|x_{1}-x_{2}|, \nonumber \\
                                                  \big|\gamma(t,x_{1},u,\eta)-\gamma(t,x_{2},u,\eta)\big|  &  \leq  &  C_{M}|\eta||x_{1}-x_{2}|, \nonumber \\
                                                  \big|\gamma(t, x,\eta,u)\big|                            &  \leq  &  C_{M}|\eta|(1+|x|). \nonumber 
                                               \end{eqnarray} 
Here and after, we will use $C$ to denote a generic constant and may differ from one line to the other. Subscripts may be added to $C$ to emphasize dependence
on particular parameters. \\

It is well known that under \hyperlink{assump1}{Assumption 1} and the compactness of $A$, there exists a constant $C>0$ such that for all $t \in [s,T]$, $u \in A$, 
and $x \in \mathbb{R}^{n}$, the coefficients $\sigma$ and $b$ satisfy 
\begin{eqnarray}
  \big|\sigma(t,x,u)\big| + \big|b(t,x,u)\big| &  \leq  & C(1+|x|).  \nonumber
\end{eqnarray}  
Moreover, there exists a unique c\`{a}dl\`{a}g and adapted solution of SDE (\ref{eqn:dynsdelarge}). To emphasize dependence on initial conditions and the control, 
we may write $X_{t}$ as $X^{u,s,x}_{t}$. \\

The revenue functional for a given $u \in \mathcal{A}_{s}$ is defined as 
\begin{equation}\label{eqn:genrf}
   V^{u}(s,x) =  \mathbb{E}\Big(\int_{s}^{T}f(t,X^{u,s,x}_{t},u_{t})dt + h(X^{u,s,x}_{T}) \Big),
\end{equation}
where $f: \mathbb{R} \times \mathbb{R}^{d} \times \mathbb{R}^{\ell} \rightarrow \mathbb{R}$ and $h: \mathbb{R}^{d} \rightarrow \mathbb{R}$ are continuous bounded 
functions. We will say that                                       
\begin{equation}\label{eqn:valuefunc}
  V(s,x) = \sup_{u \in \mathcal{A}_{s}} V^{u}(s,x)
\end{equation}
is the value function. If there exists a maximizer $u^{\ast}(s) : = u^{\ast} \in \mathcal{A}_{s}$, then
\begin{eqnarray*}
  V(s,x)  &   =  & V^{u^{\ast}}(s,x).
\end{eqnarray*}

For $s \in [0,T]$, let $\mathcal{T}_{[s,T]}$ be the set of stopping times in $[s,T]$ adapted to $(\mathcal{F}^{s}_{t})_{t \geq s}$. The DPP is then stated in the 
following Theorem.
\begin{theorem}\label{thm:bellmanjump2mr}
  \textbf{(Dynamic Programming Principle)}: For every $\tau \in \mathcal{T}_{[s,T]}$ and all $x \in \mathbb{R}^{d}$, 
  \begin{equation}\label{eqn:dpjump}
    V(s,x) = \sup_{u \in \mathcal{A}_{s}}\mathbb{E}\Big(\int_{s}^{\tau}f(r,X^{u,s,x}_{r},u_{r})dr + V(\tau,X^{u,s,x}_{\tau}) \Big).
  \end{equation}
\end{theorem}

In order to prove the DPP, we need some preparations.

\section{Auxiliary Results}
In this subsection, we present an approximation of the state process. Let $\tau_{0} = s$, and for $k = 1, 2, ...$, let $\tau_{k}$ be the arrival 
time of $k$th jump of a compound Poisson process $(L_{t})_{t \geq 0}$ after $\tau_{0}$, where
\begin{eqnarray}
   L_{t} = \int_{|\eta| \geq 1} \eta N(t,d\eta). \nonumber
\end{eqnarray}
Then, it is easy to verify that the following lemma holds.
\begin{lemma}\label{lemma:pwconvergence}
   For $M \geq 1$, let $\tau_{M}$ be a stopping time such that 
   \begin{eqnarray}
      \tau_{M} = \inf\{t > s: \Delta L_{t} \in E_{M}\},    \nonumber 
   \end{eqnarray}
   where $\Delta L_{t} = L_{t} - L_{t-}$, and $E_{M} = \{\eta\in\mathbb{R}^{q}_{0}: |\eta| \geq M\}$. As $M \rightarrow \infty$, we have 
   $1_{\{\tau_{M} \leq T\}} \rightarrow 0$, $\mathbb{P}$-a.s. In particular, we have $1_{\{\tau_{M} \leq \tau\}} \rightarrow 0$ $\mathbb{P}$-a.s. 
   for every $\tau \in \mathcal{T}_{[s,T]}$.
\end{lemma}

Set $\zeta^{M}_{0} = x$, and for $k = 1, 2, ... $, define 
\begin{eqnarray}
  \zeta^{M}_{k} &   =   &    X^{u,\tau_{k-1},\zeta^{M}_{k-1}}_{\tau_{k}}1_{\{ |\Delta L_{\tau_{k}}| < M \}} 
                           + X^{u,\tau_{k-1},\zeta^{M}_{k-1}}_{\tau_{k}-}1_{\{ |\Delta L_{\tau_{k}}| \geq M \}},  \nonumber
\end{eqnarray}
and
\begin{equation}\label{eqn:approxistate}
  X^{M}_{t} = \sum^{\infty}_{k=0}X^{u,\tau_{k},\zeta_{k}^{M}}_{t}1_{[\tau_{k},\tau_{k+1})}(t)1_{[s,T]}(t).
\end{equation}
By construction of solution, we see that $(X^{M}_{t})_{t \geq s}$ satisfies the following SDE: 
\begin{eqnarray}\label{eqn:dynsdemjumps}
  \left\{\begin{array}{ll} dX^{M}_{t} =    b(t,X^{M}_{t-},u_{t})dt + \sigma(t,X^{M}_{t-},u_{t})dW_{t} 
                                         + \displaystyle\int_{0 < |\eta| < 1}\gamma(t,X^{M}_{t-},u_{t},\eta)\tilde{N}(dt,d\eta)  \\
                                         \ \ \ \ \ \ \ \ \  + \displaystyle\int_{1 \leq |\eta| < M}\gamma(t,X^{M}_{t-},u_{t},\eta)N(dt,d\eta)   \\
                            \ \ \ \ \ \ \ \ \ \ \ \ \ \ \ \ \ \ \ \ \ \ \  \ \ \ \ \ \ \ \ \ \ \ \ \ \ \ \ \ \ \ \ \  \ \ \ \ \ \ \ \ \ \ \                      \\
                            X^{M}_{s} = x,  \ \ \ \ \ \ \ \ \ \ \ \ \ \ \ \ \ \ \ \ \ \ \ \ \ \  \ \ \ \ \ \ \ \ \ \ \ \ \ \ \ \ \ \ \ \ \ 
                                            \ \ \ \ \ \ \ \ \ \ \ \ \ \ \ 0 \leq s \leq t \leq T. \end{array}\right.
\end{eqnarray}
Again, to emphasize dependence on initial conditions and the control, we may write $X^{M}_{t}$ as $X^{u,s,x,M}_{t}$. \\

Following a standard argument, for example similar as in \cite{Kun04} (see pp340-341 in \cite{Kun04}), we can obtain the estimates below.
\begin{lemma}\label{lemma:estaproxi}
   For every $M \geq 1$, and all $p \geq 2$, there exists a $C_{T,p,M}  > 0$ such that
   \begin{enumerate}
      \item $\displaystyle\mathbb{E}\Big(\sup_{t \in [s,T]}\big|X^{u,s,x,M}_{t}\big|^{p}\Big) 
             \leq C_{T,p,M}\Big(1+|x|^{p}\Big)$,
      \item $\displaystyle\mathbb{E}\Big(\sup_{t \in [s,T]}\big|X^{u,s,x,M}_{t}-X^{u,\hat{s},\hat{x},M}_{t}\big|^{p}\Big) \leq C_{T,p,M}
             \Big(|x - \hat{x}|^{p}+ \big(1 +|x|^{p}\big)|s-\hat{s}|\Big)$.
   \end{enumerate}
\end{lemma} 

\begin{remark}
  We may extend $X^{u,\hat{s},\hat{x},M}$ by setting $X^{u,\hat{s},\hat{x},M}_{t} = \hat{x}$ for all $t \in [s,\hat{s}]$ (see p175 in \cite{Zal11}).
\end{remark}

For the sequence of state processes $(X^{M}_{t})_{t \geq s}$, we define their corresponding revenual functionals $V^{u,M}$ by
\begin{eqnarray}
   V^{u,M}(s,x) &   =   &  \mathbb{E}\Big(\int_{s}^{T}f(t,X^{u,s,x, M}_{t},u_{t})dt + h(X^{u,s,x, M}_{T})\Big). \nonumber 
\end{eqnarray}
The value functions $V^{M}$ is given by  
\begin{eqnarray}\label{eqn:VMvaluefunction}
   V^{M}(s,x) &   =   & \sup_{u \in \mathcal{A}_{s}} V^{u,M}(s,x).  
\end{eqnarray}

Next, we obtain the following lemma.
\begin{lemma}\label{lemma:convalM}
  For every $(s,x) \in [0, T] \times \mathbb{R}^{d}$, as $M \rightarrow \infty$, $V^{M}(s,x) \rightarrow V(s,x)$.
\end{lemma}  
\begin{proof}
  Since $1_{\{\tau_{M} > T\}} X^{u,s,x,M}_{t} =  1_{\{\tau_{M} > T\}} X^{u,s,x}_{t}$ ($\mathbb{P}$-a.s.) for every $t \in [s, T]$, and $f$ and $h$ are bounded, 
  we see that for every $u \in \mathcal{A}_{s}$, we have
  \begin{eqnarray}
    V^{u,M}(s,x)   &  \leq  &    \mathbb{E}\Bigg(\int_{s}^{T}f(t,X^{u,s,x}_{t},u_{t})1_{\{\tau_{M} > T\}}dt  
                               + 1_{\{\tau_{M} > T\}}h(X^{u,s,x}_{T}) + C_{T}1_{\{\tau_{M} \leq T\}}\Bigg). \nonumber
  \end{eqnarray}
  
  As $M \rightarrow \infty$, thanks again to boundedness of $f$ and $h$, we can apply the Dominated Convergence Theorem. Thus, together with 
  the continuity of $f$ and $h$, and \hyperref[lemma:pwconvergence]{Lemma \ref*{lemma:pwconvergence}}, we obtain
  \begin{eqnarray}\label{eqn:convalM1}
    \lim_{M \rightarrow \infty}V^{u,M}(s,x) &   =   &  \mathbb{E}\Big(\int_{s}^{T}f(t,X^{u,s,x}_{t},u_{t})dt + h(X^{u,s,x}_{T})\Big) \nonumber \\
                                            &   =   & V^{u}(s,x).
  \end{eqnarray}
  This then yields
  \begin{eqnarray}
    \liminf_{M \rightarrow \infty}V^{M}(s,x)  \geq   \lim_{M \rightarrow \infty}V^{u,M}(s,x) = V^{u}(s,x). \nonumber
  \end{eqnarray}
  Taking supremum over $\mathcal{A}_{s}$, we find
  \begin{eqnarray}\label{eqn:convalM2}
    \liminf_{M \rightarrow \infty}V^{M}(s,x) &   \geq   &  V(s,x).
  \end{eqnarray}
  To show the converse inequality, we observe from (\ref{eqn:convalM1}) that for every $u \in \mathcal{A}_{s}$, $(s,x) \in [0,T]\times\mathbb{R}^{d}$,
  and all $\delta > 0$, there exists an $M_{s,x,u,\delta}$ such that for all $M > M_{s,x,u,\delta}$ we have 
  \begin{eqnarray}\label{eqn:valfunest}
     \big|V^{u,M}(s,x) - V^{u}(s,x)\big| & \leq  &  \delta. 
  \end{eqnarray}
  Using (\ref{eqn:valfunest}), we see that for every $M > M_{s,x,u,\delta}$, there exists an $\epsilon$-optimal control $u^{\epsilon,M}$ such that
  \begin{eqnarray}
    V^{M}(s,x) &   \leq   &  V^{u^{\epsilon,M},M}(s,x) + \epsilon   \nonumber \\
               &   \leq   &  V^{u^{\epsilon,M}}(s,x) + \delta  + \epsilon  \nonumber \\
               &   \leq   &  V(s,x)  + \delta  + \epsilon.  \nonumber
  \end{eqnarray}
  By first letting $M \rightarrow \infty$, we obtain
  \begin{eqnarray}\label{eqn:convalM3}
    \limsup_{M \rightarrow \infty}V^{M}(s,x) &   \leq   &  V(s,x) + \delta +\epsilon.
  \end{eqnarray}
  Since $\delta$ and $\epsilon$ are arbitrary, this yields the converse inequality. The proof is completed. 
\end{proof}

Next, we present two results which we borrowed from \cite{Zal11} (modified version of Lemma 2.3 in \cite{Zal11}). Since the author does not provide a proof, 
we prove it here in our context. \\

Now, under the assumption that $f$ and $h$ are continuous, we know that the functions $f$ and $h$ admit a joint modulus of continuity (see Lemma 2.3 in 
\hbox{\cite{Zal11}}):
\begin{eqnarray}
   \rho(\alpha, \beta) & = &  \sup_{\substack{t \in [0,T], u \in A, \\ x, \hat{x} \in \overline{B(0,\beta)}, |x-\hat{x}| \leq \alpha}}
                           \Big(|f(t,x,u) - f(t,\hat{x},u)| + |h(x) - h(\hat{x})|\Big), \nonumber 
\end{eqnarray}
such that $\displaystyle\lim_{\beta \rightarrow \infty}\lim_{\alpha \rightarrow 0}\rho(\alpha, \beta) = 0$. Thus, we have the first result below.
\begin{props}\label{prop:vfljumpgen}
   There exists constants $C > 0$ and $C_{T,p,M} > 0$, such that for every $u \in \mathcal{A}_{s}$, $(s,x), (\hat{s},\hat{x}) \in [0,T] \times \mathbb{R}^{d}$, and 
   all $p \geq 2$, $\alpha > 0$, $\beta > 0$, 
   \begin{eqnarray}
                       \big|V^{u,M}(s,x) - V^{u,M}(\hat{s},\hat{x})\big|  
       &   \leq   &    C_{T}\rho(\alpha,\beta) + C_{T,p,M}\frac{|x - \hat{x}|^{p}+(1+|\hat{x}|^{p})|s-\hat{s}|}{\alpha^p}  \nonumber \\
       &          &  + C_{T,p,M}\frac{(1+|x|^{p}+|\hat{x}|^{p})}{\beta^{p}}, \nonumber 
   \end{eqnarray}
\end{props}

\begin{proof}
   For $u \in \mathcal{A}_{s}$, and $(s,x), (\hat{s},\hat{x}) \in [0,T] \times \mathbb{R}^{d}$, we see that 
   \begin{eqnarray}\label{eqn:unifcon0}
       &          &    \big|V^{u,M}(s,x) - V^{u,M}(\hat{s},\hat{x})\big|   \nonumber  \\
       &   \leq   &    \big|V^{u,M}(s,x) - V^{u,M}(s,\hat{x})\big|  
                     + \big|V^{u,M}(s,\hat{x}) - V^{u,M}(\hat{s},\hat{x})\big| \nonumber \\
       &    =     &  (I_{1}) + (I_{2}).
   \end{eqnarray}   
   The first term in (\ref{eqn:unifcon0}) yields 
   \begin{eqnarray}
     (I_{1}) &   =     &     \big|V^{u,M}(s,x) - V^{u,M}(s,\hat{x})\big|  \nonumber \\
             &  \leq   &     \mathbb{E}\Bigg(\int_{s}^{T}\big|f(t,X^{u,s,x, M}_{t},u_{t}) - f(t,X^{u,s,\hat{x}, M}_{t},u_{t})\big|dt   
                          +  \big|h(X^{u,s,x, M}_{T}) -  h(X^{u,s,\hat{x}, M}_{T})\big|\Bigg) \nonumber \\
             &    =    &     (I_{1,1}) + (I_{1,2}). \nonumber
   \end{eqnarray}
   The second term in (\ref{eqn:unifcon0}) can be estimated as
   \begin{eqnarray}
     (I_{2}) &    =    &     \big|V^{u,M}(s,\hat{x}) - V^{u,M}(\hat{s},\hat{x})\big|  \nonumber \\
             &  \leq   &     \mathbb{E}\Bigg(\int_{s}^{\hat{s}}\big|f(t,X^{u,s,\hat{x}, M}_{t},u_{t}) 
                           - f(t,X^{u,\hat{s},\hat{x}, M}_{t},u_{t})\big|dt  
                           +  \big|h(X^{u,s,\hat{x}, M}_{T}) -  h(X^{u,\hat{s},\hat{x}, M}_{T})\big|\Bigg) \nonumber \\
             &    =    &     (I_{1,3}) + (I_{1,4}). \nonumber
   \end{eqnarray}   
   Each of $(I_{1,1}) - (I_{1,4})$ can be estimated by using the bounds of $f$ and $h$, the Markov inequality, and 
   \hyperref[lemma:estaproxi]{Lemma \ref*{lemma:estaproxi}}. For example, for $(I)$ we have   
   \begin{eqnarray}
     (I_{1,1}) &   =   &   \mathbb{E}\Bigg(\int_{s}^{T}\big|f(t,X^{u,s,x, M}_{t},u_{t}) - f(t,X^{u,s,\hat{x}, M}_{t},u_{t})\big|dt\Bigg)  \nonumber \\
               &  \leq &    C_{T}\mathbb{P}\Bigg(\sup_{t \in [s,T]}\big|X^{u,s,x, M}_{t} - X^{u,s,\hat{x}, M}_{t}\big|^{p} > \alpha^p\Bigg) 
                          + C_{T}\rho(\alpha,\beta) \nonumber       
   \end{eqnarray}
   \begin{eqnarray}
              &       &  + C_{T}\mathbb{P}\Bigg(\sup_{t \in [s,T]}\big|X^{u,s,x, M}_{t}\big|^{p} \geq \beta^{p} \Bigg) 
                         + C_{T}\mathbb{P}\Bigg(\sup_{t \in [s,T]}\big|X^{u,s,\hat{x}, M}_{t}\big|^{p} \geq \beta^{p} \Bigg) \nonumber \\
              &  \leq &    C_{T}\rho(\alpha,\beta) + C_{T,p,M}\frac{|x - \hat{x}|^{p}}{\alpha^p}
                         + C_{T,p,M}\frac{(1+|x|^{p}+|\hat{x}|^{p})}{\beta^{p}}. \nonumber 
   \end{eqnarray}
   In a similar way, we obtain
   \begin{eqnarray}
       (I_{1,2})  &  \leq   &   C_{T}\rho(\alpha,\beta) + C_{T,p,M}\frac{|x - \hat{x}|^{p}}{\alpha^p}
                              + C_{T,p,M}\frac{(1+|x|^{p}+|\hat{x}|^{p})}{\beta^{p}},  \nonumber \\
       (I_{1,3})  &  \leq   &   C_{T}\rho(\alpha,\beta) + C_{T,p,M}\frac{(1+|\hat{x}|^{p})|s-\hat{s}|}{\alpha^p} 
                              + C_{T,p,M}\frac{(1+|\hat{x}|^{p})}{\beta^{p}}, \nonumber \\
       (I_{1,4})  &  \leq   &   C_{T}\rho(\alpha,\beta) + C_{T,p,M}\frac{(1+|\hat{x}|^{p})|s-\hat{s}|}{\alpha^p} 
                              + C_{T,p,M}\frac{(1+|\hat{x}|^{p})}{\beta^{p}}. \nonumber 
   \end{eqnarray}
   Combing $(I)-(IV)$, we complete the proof. \\
\end{proof}

Since
\begin{eqnarray}\label{eqn:lipvalinitialgen}
              |V^{M}(s,x) - V^{M}(s,\hat{x})|    
  &    =   &  \Big|\sup_{u \in \mathcal{A}_{s}} V^{u,M}(s,x) - \sup_{u \in \mathcal{A}_{s}} V^{u,M}(s,\hat{x})\Big| \nonumber \\
  &  \leq  &  \sup_{u \in \mathcal{A}_{s}} \Big|V^{u,M}(s,x) - V^{u,M}(s,\hat{x})\Big|, \nonumber 
\end{eqnarray}
the following corollary is a direct consequence of \hyperref[prop:vfljumpgen]{Proposition \ref*{prop:vfljumpgen}}. 
\begin{corollary}\label{corollary:unifvaluefunM}
    For all $p \geq 2$, there exists constants $C_{T} > 0$ and $C_{T,p,M} > 0$ such that for $\alpha > 0$, $\beta > 0$, and $(s,x), (\hat{s},\hat{x}) \in [0,T] \times 
   \mathbb{R}^{d}$, we have
   \begin{eqnarray}
       |V^{M}(s,x) - V^{M}(\hat{s},\hat{x})|   &  \leq  &    C_{T}\rho(\alpha,\beta) 
                                                           + C_{T,p,M}\frac{|x - \hat{x}|^{p}+(1+|\hat{x}|^{p})|s-\hat{s}|}{\alpha^p} \nonumber \\
                                               &        &  + C_{T,p,M}\frac{(1+|x|^{p}+|\hat{x}|^{p})}{\beta^{p}}. \nonumber 
   \end{eqnarray} 
\end{corollary}

In order to prove the DPP, the Markov characterization of the state process (see for example, Lemma 3.2 in \cite{Zal11}) plays an important role. 
The next lemma states the controlled Markovian property for jump processes. 

\begin{lemma}\label{lemma:eqn:markovstate} 
The following two assertions hold.
 \begin{enumerate}
    \item For almost every $\omega \in \Omega$, all $\tau \in \mathcal{T}_{[s,T]}$, and $u \in \mathcal{A}_{s}$, there exists a control $\hat{u}^{\omega} 
    \in \mathcal{A}_{\tau}$ such that 
    \begin{eqnarray}
             \mathbb{E}\Big(\int_{\tau}^{T}f(r,X^{u,s,x,M}_{r},u_{r})dr + h(X^{u,s,x,M}_{T})\big|\mathcal{F}_{\tau}\Big)(\omega) 
           = V^{\hat{u}^{\omega},M}(\tau(\omega),X^{u,s,x,M}_{\tau}(\omega)).  \nonumber 
    \end{eqnarray}   
    \item For every $t \in [s,T]$, and all $\tau \in \mathcal{T}_{[s,t]}$, there exists a control $\hat{u} \in \mathcal{A}_{s}$, where 
          \begin{eqnarray}
            \hat{u}_{r} := u_{r}1_{\{r \in [s,\tau]\}} + \tilde{u}_{r}1_{\{r \in (\tau,T]\}}, \nonumber 
          \end{eqnarray}
          and $\tilde{u} \in \mathcal{A}_{t}$, such that
          \begin{eqnarray}
              \mathbb{E}\Big(\int_{\tau}^{T}f(r,X^{\hat{u},s,x,M}_{r},u_{r})dr + h(X^{\hat{u},s,x,M}_{T})\big|\mathcal{F}_{\tau}\Big)(\omega) 
            = V^{\tilde{u},M}(\tau(\omega),X^{u,s,x,M}_{\tau}(\omega)) \ \ \mathbb{P}\textnormal{-a.s.}  \nonumber 
          \end{eqnarray}
 \end{enumerate}
\end{lemma}
\begin{proof}
   The proof follows from Remark 3.10 and the proof of Proposition 5.4 in \cite{BouT11}. \\
\end{proof}

\section{The Proof of DPP}
We now proceed to the prove of DPP. We will follow \cite{BouT11} and \cite{Zal11}.

\begin{proof}
   We start from the easy direction. For $\tau \in \mathcal{T}_{[s,T]}$, $u \in \mathcal{A}_{s}$, and by the first assertion of 
   \hyperref[lemma:eqn:markovstate]{Lemma \ref*{lemma:eqn:markovstate}},  we see that for $M \geq 1$ there exists a control $\hat{u} \in 
   \mathcal{A}_{\tau}$ such that
   \begin{eqnarray}
      &         &   \mathbb{E}\Big(\int_{s}^{T}f(t, X^{u,s,x,M}_{t},u_{t})dt + h(X^{u,s,x,M}_{T})\Big)       \nonumber \\             
      &    =    &   \mathbb{E}\Big(\int_{s}^{\tau}f(t,X^{u,s,x,M}_{t},u_{t})dt + V^{\hat{u},M}(\tau,X^{u,s,x,M}_{\tau})\Big) \nonumber \\
      &   \leq  &   \mathbb{E}\Big(\int_{s}^{\tau}f(t,X^{u,s,x}_{t},u_{t})1_{\{\tau_{M} > \tau\}}dt + C_{T}1_{\{\tau_{M} \leq \tau\}} 
                   + V^{M}(\tau,X^{u,s,x}_{\tau})1_{\{\tau_{M} > \tau\}}\Big). \nonumber
   \end{eqnarray}
   In the last line, we have used the fact that $1_{\{\tau_{M} > \tau\}} X^{u,s,x,M}_{t} =  1_{\{\tau_{M} > \tau\}} X^{u,s,x}_{t}$ 
   ($\mathbb{P}$-a.s.) for every $t \in [s, \tau]$. As $M \rightarrow \infty$, by boundedness of $f$, and $h$, we can apply the Dominated 
   Convergence Theorem. Thus, together with \hyperref[lemma:pwconvergence]{Lemma \ref*{lemma:pwconvergence}} and 
   \hyperref[lemma:convalM]{Lemma \ref*{lemma:convalM}}, we have
   \begin{eqnarray}
                    \mathbb{E}\Big(\int_{s}^{T}f(t, X^{u,s,x}_{t},u_{t})dt + h(X^{u,s,x}_{T})\Big)                     
      &   \leq  &   \mathbb{E}\Big(\int_{s}^{\tau}f(t,X^{u,s,x}_{t},u_{t})dt + V(\tau,X^{u,s,x}_{\tau})\Big). \nonumber 
   \end{eqnarray}
   Taking supremum over $\mathcal{A}_{s}$, we obtain  
   \begin{eqnarray}\label{eqn:valfunMdpp1}
        V(s,x)   &  \leq   &   \sup_{u \in \mathcal{A}_{s}}\mathbb{E}\Big(\int_{s}^{\tau}f(t,X^{u,s,x}_{t},u_{t})dt + V(\tau,X^{u,s,x}_{\tau})\Big). 
   \end{eqnarray}
    
   To show the converse, fix $\epsilon \in (0, 1)$ and choose $\alpha < \epsilon$. Next, choose $\beta > (\frac{1}{\epsilon})^{\frac{1}{p}}$ such that 
   $\rho(\alpha,\beta) < \epsilon$. For a fixed $p \geq 2$, let us take a Borel partition $\{B_{j}\}_{j \geq 1}$ of $\mathbb{R}^{d}$ such that 
   \begin{eqnarray}\label{eqn:assumpxdpps}
      \displaystyle\sup_{x_{j}, \hat{x}_{j} \in B_{j}}|x_{j}-\hat{x}_{j}|^{p} \leq \alpha^{p}\epsilon. 
   \end{eqnarray}
   For $M \geq 1$, $t \in [s,T]$ and $x \in \mathbb{R}^{d}$, we know that there exists an $\epsilon$-optimal control $\tilde{u}^{\epsilon,M} \in \mathcal{A}_{t}$ 
   such that   
   \begin{eqnarray}\label{eqn:VMep}
      V^{M}(t,x) \leq V^{M,\tilde{u}^{\epsilon,M}}(t,x) + \epsilon.
   \end{eqnarray}
   By \hyperref[corollary:unifvaluefunM]{Corollary \ref*{corollary:unifvaluefunM}}, (\ref{eqn:assumpxdpps})-(\ref{eqn:VMep}), and
   \hyperref[prop:vfljumpgen]{Proposition \ref*{prop:vfljumpgen}}, we see that for every $x_{j}, \hat{x}_{j} \in B_{j}$, there exists 
   an $\epsilon$-optimal control $\tilde{u}^{j,\epsilon,M} \in \mathcal{A}_{t}$ such that
   \begin{eqnarray}\label{eqn:3epsilon}
      V^{M}(t,x_{j})   &    \leq   &   V^{M}(t,\hat{x}_{j}) + \epsilon C_{T,p,M}(1+|x_{j}|^{p}) \nonumber \\
                       &    \leq   &   V^{M,\tilde{u}^{j,\epsilon,M}}(t,x_{j}) + \epsilon C_{T,p,M}(1+|x_{j}|^{p}) + \epsilon \nonumber \\
                       &    \leq   &   V^{M,\tilde{u}^{j,\epsilon,M}}(t,x_{j}) + \epsilon C_{T,p,M}(1+|x_{j}|^{p}).
   \end{eqnarray}
   For $u \in \mathcal{A}_{s}$, we take a sequence of controls
   \begin{eqnarray}
      \hat{u}^{j,\epsilon,M}_{r} = \left\{\begin{array}{ll}  u_{r},                          &   \textnormal{if} \ r \in [s,t],     \\ 
                                                             \tilde{u}^{j,\epsilon,M}_{r},   &   \textnormal{if} \ r \in (t,T] \ \textnormal{and} 
                                                                                           \ X^{u,s,x,M}_{t} \in B_{j},  \end{array}\right. \nonumber 
   \end{eqnarray}
   where $\tilde{u}^{j,\epsilon,M} \in \mathcal{A}_{t}$. It is easy to see that $\hat{u}^{j,\epsilon,M} \in \mathcal{A}_{s}$ which is a consequence 
   of the measurability of $X^{u,s,x,M}_{t}$ and the fact that $\mathcal{F}^{t}_{r} \subset \mathcal{F}^{s}_{r}$ for all $s \leq t \leq r$. By 
   uniqueness of solution, the second assertion of \hyperref[lemma:eqn:markovstate]{Lemma \ref*{lemma:eqn:markovstate}} and (\ref{eqn:3epsilon}) we 
   then obtain 
   \begin{eqnarray} 
                    V^{M}(s,x) 
      &  \geq  &    \mathbb{E}\Bigg(\int^{T}_{s}f(r,X^{\hat{u}^{j,\epsilon,M},s, x, M}_{r},u_{r})dr + h(X^{\hat{u}^{j,\epsilon,M},s, x, M}_{T}) \Bigg) \nonumber \\
      &   =    &   \mathbb{E}\Bigg(\int^{t}_{s}f(r,X^{u,s, x, M}_{r},u_{r})dr\Bigg) 
                  + \sum_{j \geq 1}\mathbb{E}\Bigg(\mathbb{E}\Big(\int^{T}_{t}f(r,X^{\tilde{u}^{j,\epsilon,M}, t, X^{u,s,x,M}_{t}}_{r},\tilde{u}^{j,\epsilon,M}_{r})dr \nonumber 
   \end{eqnarray}
   \begin{eqnarray}
      &        &  + h(X^{\tilde{u}^{\epsilon,j,M}, t, X^{u,s,x,M}_{t}}_{T}) \big| \mathcal{F}_{t} \Big)1_{\{X^{u,s,x,M}_{t} \in B_{j}\}}\Bigg) \nonumber \\
      &   =    &   \mathbb{E}\Bigg(\int^{t}_{s}f(r,X^{u,s, x, M}_{r},u_{r})dr\Bigg) 
                  + \sum_{j \geq 1}\mathbb{E}\Big(V^{u^{\epsilon,j,M},M}(t,X^{u,s,x,M}_{t})1_{\{X^{u,s,x,M}_{t} \in B_{j}\}}\Big) \nonumber  \\
      &  \geq  &    \mathbb{E}\Bigg(\int^{t}_{s}f(r,X^{u,s, x, M}_{r},u_{r})dr\Bigg) + \sum_{j \geq 1}\mathbb{E}\Bigg(\Big(V^{M}(t,X^{u,s,x,M}_{t}) \nonumber \\
      &        &   - \epsilon C_{T,p,M}(1+|X^{u,s,x,M}_{t}|^{p})\Big)1_{\{X^{u,s,x,M}_{t} \in B_{j}\}}\Bigg)  \nonumber \\ 
      &   =    &    \mathbb{E}\Bigg(\int^{t}_{s}f(r,X^{u,s, x, M}_{r},u_{r})dr + V^{M}(t,X^{u,s,x,M}_{t})\Bigg) 
                   - \epsilon C_{T,p,M}\mathbb{E}\Big(1+|X^{u,s,x,M}_{t}|^{p}\Big).  \nonumber
    \end{eqnarray}     
   Since $\epsilon$ is arbitrary,  we then have  
   \begin{eqnarray} 
        V^{M}(s,x)    &  \geq  &  \mathbb{E}\Big(\int^{t}_{s}f(r,X^{u,s, x, M}_{r},u_{r})dr + V^{M}(t,X^{u,s,x,M}_{t})\Big). \nonumber 
   \end{eqnarray}    
   Let $\mathcal{G}(t) :=  \int_{s}^{t}f(X^{u,s,x}_{r})dr + V^{M}(t,X^{u,s,x}_{t}$). For every $t_{1}, t_{2} \in [s,T)$, and 
   
   \begin{eqnarray}
        \hat{u}_{r} = \left\{\begin{array}{ll}  u_{r},             &   \textnormal{if} \ r \in [s,t_{1}],     \\ 
                                                \tilde{u}_{r},     &   \textnormal{if} \ r \in (t_{1},T],  \end{array}\right. \nonumber 
   \end{eqnarray}
   where $u \in \mathcal{A}_{s}$ and $\tilde{u} \in \mathcal{A}_{t_{1}}$, we have, by uniqueness of solution and the second assertion of 
   \hyperref[lemma:eqn:markovstate]{Lemma \ref*{lemma:eqn:markovstate}}, 
   \begin{eqnarray}
                     \mathbb{E}(\mathcal{G}(t_{2})|\mathcal{F}_{t_{1}})  
       &   =   &   \int_{s}^{t_{1}}f(t,X^{u,s,x,M}_{t},u_{t})dt  \nonumber \\
       &       &  + \mathbb{E}\Bigg(\int_{t_{1}}^{t_{2}}f(t,X^{u,s,X^{u,s,x,M}_{t_{1}},u_{t}}_{t},\tilde{u}_{t})dt 
                  + V^{M}(t_{2},X^{u,s,X^{u,s,x,M}_{t_{1}}}_{t_{2}})\big|\mathcal{F}_{t_{1}}\Bigg) \nonumber   \\    
       &  \leq &   \int_{s}^{t_{1}}f(t,X^{u,s,x,M}_{t},u_{t})dt + V^{M}(t_{1},X^{u,s,x,M}_{t_{1}})   \nonumber \\
       &   =   &   \mathcal{G}(t_{1}). \nonumber 
   \end{eqnarray}
   Thus, $\mathcal{G}$ is a supermartingale, and by Doob's Optional Sampling Theorem we know that, for every stopping time $\tau \in \mathcal{T}_{[s,T]}$ 
   and $u \in \mathcal{A}_{s}$, we have
   \begin{eqnarray}\label{eqn:valfunMdpp2}
     V^{M}(s,x)  &  \geq  & \mathbb{E}\Big(\int_{s}^{\tau}f(t,X^{u,s,x,M}_{t},u_{t})dt + V^{M}(\tau,X^{u,s,x,M}_{\tau})\Big). 
   \end{eqnarray}
   Without loss of generality, we assume that $f, h > 0$ for all $(s,x) \in [s,T]\times\mathbb{R}^{d}$. Then, (\ref{eqn:valfunMdpp2}) implies
   \begin{eqnarray}
     V^{M}(s,x) &  \geq  &  \mathbb{E}\Big(\int_{s}^{\tau}f(t,X^{u,s,x}_{t},u_{t})1_{\{\tau_{M} > \tau\}}dt 
                           + V^{M}(\tau,X^{u,s,x}_{\tau})1_{\{\tau_{M} > \tau\}} \Big). \nonumber 
   \end{eqnarray}
   Here, we use the fact that $1_{\{\tau_{M} > \tau\}} X^{u,s,x,M}_{t} =  1_{\{\tau_{M} > \tau\}} X^{u,s,x}_{t}$ ($\mathbb{P}$-a.s.) for every
   $t \in [s, \tau]$. As $M \rightarrow \infty$, thanks to the boundedness of $f$ and $h$, the Dominated Convergence Theorem can be applied. 
   Together with \hyperref[lemma:pwconvergence]{Lemma \ref*{lemma:pwconvergence}} and \hyperref[lemma:convalM]{Lemma \ref*{lemma:convalM}}, 
   the above yields 
   \begin{eqnarray}
      V(s,x) &  \geq  &  \mathbb{E}\Big(\int_{s}^{\tau}f(t,X^{u,s,x}_{t},u_{t})dt + V(\tau,X^{u,s,x}_{\tau})\Big). \nonumber 
   \end{eqnarray}    
   Taking supremum over $\mathcal{A}_{s}$, and combining with (\ref{eqn:valfunMdpp1}) we obtain the desired result. 
\end{proof}






%
%
%

\end{document}